\documentclass[a4paper,12pt,reqno]{amsart}
\usepackage{tikz}
\usetikzlibrary{arrows,shapes}
\usepackage{amsmath}%
\usepackage{amsfonts}%
\usepackage{amssymb}%
\usepackage{graphicx}
\usepackage{float} %设置图片浮动位置的宏包
\usepackage{subfigure} %插入多图时用子图显示的宏包
\usepackage{enumerate}
\usepackage{lineno,hyperref}
\hypersetup{
	colorlinks = true,
	linkcolor = blue,
	anchorcolor = blue,
	citecolor = blue,
	filecolor = blue,
	urlcolor = blue
}
\newcommand{\C}{\mathbb{C}}
\newcommand{\D}{\mathbb{D}}
\newcommand{\N}{\mathbb{N}}
\newtheorem{theorem}{Theorem}
\theoremstyle{plain}

\newtheorem{definition}{Definition}

\newtheorem{lemma}{Lemma}

\newtheorem{proposition}{Proposition}
\newtheorem{remark}{Remark}

\numberwithin{equation}{section}
\begin{document}
\title[Uniqueness for totally geodesic hypersurfaces]{A Uniqueness Theorem for Holomorphic Mappings in the Disk Sharing Totally Geodesic Hypersurfaces}
\author[Huang]{Jiaxing Huang}
\address{College of Mathematics and Statistics, Shenzhen University, Guangdong, 518060, P. R. China}
\email{hjxmath@szu.edu.cn}
\author[Ng]{Tuen Wai Ng}
\address{Department of Mathematics, The University of Hong Kong, Pokfulam, Hong Kong}
\email{ntw@maths.hku.hk}
\thanks{Jiaxing Huang was partially supported by a graduate studentship of HKU, the RGC grant (Grant No. 1731115) and the National Natural Science Foundation of China (Grant No. 11701382).  Tuen Wai Ng was partially supported by the RGC grant (Grant No. 1731115 and 17307420).}

\begin{abstract}In this paper, we prove a Second Main Theorem for holomorphic mappings in a disk whose image intersects some families of nonlinear hypersurfaces (totally geodesic hypersurfaces with respect to a meromorphic connection) in the complex projective space $\mathbb{P}^k$. This is a generalization of Cartan's Second Main Theorem. As a consequence, we establish a uniqueness theorem for holomorphic mappings which intersects $O(k^3)$ many totally geodesic hypersurfaces.
\end{abstract}
\maketitle

\section{Introduction and main results}\label{sec:1}
Let $f$ and $g$ be two nonconstant meromorphic functions defined on $\mathbb{C}$. We say that $f$ and $g$ share a complex number $b$ CM (IM) if $f-b$ and $g-b$ have the same zeros in $\mathbb{C}$, counting with the same multiplicities (ignoring multiplicities). In 1926, Nevanlinna proved a uniqueness theorem for meromorphic functions, based on his Second Main Theorem (SMT). This is the famous Nevanlinna Five Point Theorem, which says that any two nonconstant meromorphic functions $f$ and $g$ sharing five values IM must be identical.  As each meromorphic function could be seen as a holomorphic curve from $\C$ into the projective space $\mathbb{P}^1$, it is natural to extend Nevanlinna's uniqueness theorem to holomorphic curves in higher dimensional complex projective spaces. 

In 1974, Fujimoto \cite{Fuj75} showed that for two linearly non-degenerated meromorphic mappings $f$ and $g$ from $\mathbb{C}$ into $\mathbb{P}^k$, if they have the same inverse images of $3k+2$ hyperplanes counted with multiplicities in $\mathbb{P}^k$ in a general position, then $f\equiv g$. In 1983, Smiley \cite{Smil83} obtained the uniqueness theorem of the meromorphic mappings sharing $3k+2$ hyperplanes without counting multiplicity. Later on, in 2009, Chen and Yan \cite{CY09} reduced the number $3k+2$ to $2k+3$, which is the smallest number of sharing hyperplanes in the uniqueness theorems so far; this type of result was extended and developed by many authors, such as those of \cite{DQT12}, and others. For example, one may naturally consider the problem of sharing hypersurfaces. In 2008, Dulock and Ru \cite{DR08} investigated  for the first time the case of sharing hypersurfaces without counting multiplicities, and they proved the following theorem by using a result of An and Phuong in \cite{AP09}:

\begin{theorem}[Dulock-Ru \cite{DR08}]\label{thm:DMR}
	Let $\{Q_j\}_{j=1}^q $ be hypersurfaces of degree $d_j$ in $\mathbb{P}^k$ in a general position. Let $d_0=\min\{d_1, \dots, d_q\}$, $d=\mathrm{lcm}\{d_1, \dots, d_q\}$, and $M=2d[2^{k-1}(k+1)kd(d+1)]^k$. Suppose that $f$ and $g$ are algebraically non-degenerated meromorphic mappings of $\mathbb{C}$ into $\mathbb{P}^k$ such that
	$f(z)=g(z)$ for any $z\in\displaystyle\bigcup_{j=1}^q \left\{f^{-1}(Q_j)\cup g^{-1}(Q_j)\right\}$. If $q>(k+1)+\dfrac{2M}{d_0} +\dfrac{1}{2}$,
	then $f\equiv g$. Here $f^{-1}(Q_j)$ means the zero set of $Q_j\circ f$.
\end{theorem}

The proof of Theorem \ref{thm:DMR} is based on a Second Main Theorem of An and Phuong \cite{AP09}, and also the following lemma of Dulock and Ru \cite{DR10}:
\begin{lemma}[\cite{DR10}]\label{lem:DMR}
	Let $H$ be a hyperplane line bundle on $\mathbb{P}^k$. For $m=1, 2$, we let $\pi_m:\mathbb{P}^k\times\mathbb{P}^k\rightarrow\mathbb{P}^k$ be the canonical projection mappings.
	Let $f\times g:\mathbb{C}\rightarrow\mathbb{P}^k\times\mathbb{P}^k$ be a holomorphic map such that $f\not\equiv g$, there exists a section $s$ of $H':=\pi_1^*H\otimes\pi_2^*H$ so that the diagonal $\Delta$ of $\mathbb{P}^k\times\mathbb{P}^k$ is contained in $\mathrm{Supp}(s)$, but the image $(f\times g)(\mathbb{C})$ is not contained in $\mathrm{Supp}(s)$.
\end{lemma}

We notice that the number of sharing hypersurfaces in Dulock-Ru's result is of order $k^{2k}$, which is much bigger than $3k+2$ or $2k+3$ in the hyperplane case. As an improvement of the truncated version of Ru's Second Main Theorem \cite{Ru09}, an expected smaller number of hypersurfaces can be found in Theorem \ref{thm:QandA} (below) of Quang and An \cite{QA17}. However, this number is still much bigger than $3k+2$ or $2k+3$.  Therefore, it would be interesting to try to get a uniqueness theorem for holomorphic curves sharing fewer hypersurfaces. 

Let $V$ be a complex projective subvariety of $\mathbb{P}^k$ of dimension $m(\leq k)$. Let $d$ be a positive integer. We denote by $I(V)$ the ideal of homogeneous polynomials in $\mathbb{C}[X_0,\dots, X_k]$ defining $V$,  and by $H_d$ the $\mathbb{C}$-vector space of all homogeneous polynomials in $\mathbb{C}[X_0,\dots, X_k]$ of degree $d$. Define $$I_d(V)=\frac{H_d}{I(V)\cap H_d}\ \  \mbox{and}\ \ H_V(d)=\mbox{dim}\  I_d(V).$$ Then $H_V(d)$ is the Hilbert function of $V$. Each element of $I_V(d)$ can be represented by $[Q]$ for some $Q\in H_d$. In the case where $V$ is a linear space of dimension $m$ and $d=1$, we have that $H_V(d)=m+1$. For the general case, it is easy to see that $H_V(d)\leq\dfrac{(k+d)!}{k!d!}\sim \dfrac{(k+d)^d}{d!}$ for a fixed $d$ and $k$ sufficiently large.

\begin{definition}
	Let $f:\mathbb{C}\rightarrow V$ be a holomorphic mapping of $\mathbb{C}$ into $V$. Then $f$ is said to be degenerated over $I_d(V)$ if there exists a non-zero $[Q]\in I_d(V)$ such that $Q(f)\equiv 0$. Otherwise, we say that $f$ is non-degenerated over $I_d(V)$. One can see that if $f$ is algebraically non-degenerated, then $f$ is non-degenerated over $I_d(V)$ for $d\geq 1$.
\end{definition} 

In 2017, Quang and An first established a truncated version of the Second Main Theorem involving $H_V(d)$ and as an application of this, they improved Dulock-Ru's result (Theorem \ref{thm:DMR}) and obtained the following uniqueness theorem for holomorphic curves sharing a possibly smaller number of hypersurfaces in $\mathbb{P}^k$:
\begin{theorem}[\cite{QA17}]\label{thm:QandA}
	Let $V$ be a complex projective subvariety of $\mathbb{P}^k$ of dimension $m (m\leq k)$. Let $\{Q_i\}_{i=1}^q$ be hypersurfaces in $\mathbb{P}^k$ in an $N$-subgeneral position with respect to $V$ and deg $Q_i=d_i$ $(1\leq i\leq q)$. Let $d=\mathrm{lcm}(d_1, \dots, d_q)$. Let $f$ and $g$ be holomorphic curves of $\mathbb{C}$ into $V$ which are non-degenerated over $I_d(V)$. Assume that 
	$f(z)=g(z)$ for any $z\in\displaystyle\bigcup_{i=1}^q \left\{f^{-1}(Q_j)\cup g^{-1}(Q_j)\right\}$.
	Then the following assertions hold:
	\begin{enumerate}[(a)]
		\item if $q> \dfrac{2(H_V(d)-1)}{d}+\dfrac{(2N-m+1)H_V(d)}{m+1},$ then $f\equiv g$;
		\item if $q> \dfrac{2(2N-m+1)H_V(d)}{m+1},$ then there exist $N+1$ hypersurfaces $(Q_{i_0}), \dots, (Q_{i_N}), 1\leq i_0 <\cdots < i_N\leq q$ such that $$\frac{Q_{i_0}(f)}{Q_{i_0}(g)} =\cdots = \frac{Q_{i_N}(f)}{Q_{i_N}(g)}.$$
	\end{enumerate}
	
\end{theorem}

\begin{remark} Part (b) of Theorem \ref{thm:QandA} implies Chen-Yan's result \cite{CY09} (see Corollary 1 in \cite{QA17}).
\end{remark}

Although the number of sharing hypersurfaces in Theorem \ref{thm:QandA} is much smaller than the one in Dulock-Ru's result (Theorem \ref{thm:DMR}), the number $H_V(d)$ is not easy to explicitly estimate and is bounded by $O(k^d)$ depending on the degree $d$ of the hypersurfaces. In this paper, we would like to give an explicit estimation (around $O(k^3)$ at independent over the degree $d$) of the number of shared special hypersurfaces. 

So far, the tools to solve the unicity problem of holomorphic curves have been various versions of the Second Main Theorem. In 2012, Tiba \cite{Tib12} made use of Demailly's \cite{Dem97} meromorphic partial projective connection (see Definition \ref{def:mppc}) to prove a Second Main Theorem for a holomorphic curve in $\mathbb{P}^k$ crossing totally geodesic hypersurfaces (see Definition \ref{def:tgh}). As a consequence, one can obtain a uniqueness theorem of holomorphic curves intersecting totally geodesic hypersurfaces; the required number of hypersurfaces is smaller than the one in Dulock-Ru's result (Theorem \ref{thm:DMR}), and more precise than the one in Quang-An's result (Theorem \ref{thm:QandA}).

To formulate our result, we have to introduce the definition of meromorphic partial projective connections first provided by Siu \cite{Siu87}, and that of totally geodesic hypersurfaces on a complex projective algebraic manifold $X$.  One can refer to Demaily \cite{Dem97}, Section 11 or Tiba \cite{Tib12}, Section 3 for the details.
%Hence, from  the proof of the results of Dulock-Ru and Quang-An, we consider the unicity of holomorphic curves sharing few totally geodesic hypersurfaces. 

Let $\{U_j\}_{1\leq j\leq N}$ be an affine open covering of $X$. 
\begin{definition}[Meromorphic partial projective connection]\label{def:mppc}
	A meromorphic partial projective connection $\Lambda$, relative to an affine open covering $\{U_j \}_{1\leq j \leq N}$ of $X$, is a collection of meromorphic connections $\Lambda_j$ on $U_j$ satisfying
	$$\Lambda_j -\Lambda_k =\alpha_{jk}\otimes Id_{TX} + Id_{TX} \otimes \beta_{jk}$$ on $U_j \cap U_k$
	for all $1 \leq j,k \leq N$, where $\alpha_{jk}$ and $\beta_{jk}$ are meromorphic one-forms on $U_j \cap U_k$. We write
	$\Lambda = \{(\Lambda_j,U_j)\}_{1\leq j\leq N}.$
\end{definition}
Let $S_j$ be the smallest subvariety of $X$ such that $\Lambda_j$ is a holomorphic connection on $U_j\setminus S_j \cap U_j$. We set supp$(\Lambda)_{\infty}:= \bigcup_{1\leq j\leq N} S_j$ and call it the polar locus of $\Lambda$.

Let $D$ be a reduced effective divisor of a $k$-dimensional complex projective algebraic manifold $X$ , and let $\Lambda$ be a meromorphic connection. Consider the holomorphic function $s$ on an open set $U \subset X$ such that $D|_U = (s)$, and fix a local coordinate system $(z_1,\dots,z_n)$ on $U$. In particular, if $X=\mathbb{P}^k$, one can always construct a meromorphic partial projective connection from some given homogenous polynomials (see Demailly \cite{Dem97} or Tiba \cite{Tib12}).

Let $S_0, \dots, S_k$ be homogenous polynomials of degree $d$ in $\C[X_0, \dots, X_k]$ such that   
\begin{equation*}
	\det\left(\frac{\partial S_{\mu}}{\partial X_i}\right)_{0\leq \mu, i\leq k}\not\equiv 0.
\end{equation*} 
Then one can construct the meromorphic connection $\widetilde{\Lambda}=d + \widetilde{\Gamma}$ on $\mathbb{C}^{k+1}$ defined by 
\begin{equation}\label{eqn:2}
	\sum_{0\leq\lambda\leq k}\frac{\partial S_{\mu}}{\partial X_{\lambda}}\widetilde{\Gamma}_{i,j}^{\lambda}=\frac{\partial^2S_{\mu}}{\partial X_i\partial X_j}.
\end{equation} 
This meromorphic connection induces the meromorphic partial projective connection $\Lambda$ on $\mathbb{P}^k$ by the following lemma from Demailly \cite{Dem97}:
\begin{lemma}[\cite{Dem97}]
	Let $\widetilde{\Lambda}=d + \widetilde{\Gamma}$ on $\mathbb{C}^{k+1}$. Let $\varsigma=\sum z_j\frac{\partial}{\partial z_j}$ be the Euler vector field on $\mathbb{C}^{k+1}$
	and let $\pi:\mathbb{C}^{k+1}\setminus\{0\}\rightarrow\mathbb{P}^k$ be the canonical projection. Then $\widetilde{\Lambda}$ induces a meromorphic partial projective connection $\Lambda$ on $\mathbb{P}^k$ provided that 
	\begin{enumerate}[(i)]
		\item the Christoffel symbols $\widetilde{\Gamma}_{j\mu}^{\lambda}$ are homogeneous rational functions of degree $-1$ (i.e. the total degree of its numerator minus that of its denominator is equal to $-1$);
		\item  there are meromorphic functions $\alpha, \beta$ and meromorphic 1-forms $\gamma, \eta$ on $\mathbb{C}^{k+1}\setminus\{0\}$ such that $\widetilde{\Gamma}\cdot<\varsigma, v>=\alpha v+\gamma(v)\varsigma$ and $\widetilde{\Gamma}\cdot<w, \varsigma>=\beta w+\eta(w)\varsigma$ for all vector fields $v, w$.
	\end{enumerate}
\end{lemma} 

\begin{definition}[Totally geodesic hypersurface]\label{def:tgh}
	A hypersurface $D$ is said to be totally geodesic with respect to $\Lambda$ on $U$ if there exist meromorphic one-forms $$a=\sum_{j=1}^k a_jdz_j,\ b =\sum_{j=1}^kb_jdz_j$$ and a meromorphic two-form
	$c =\sum_{1\leq j,\mu\leq k} c_{j\mu}dz_j \otimes dz_{\mu}$ such that no polar locus of $a_j,b_j$, or $c_{j\mu} (1 \leq j,\mu \leq k)$ contains supp $D|_U$ , and
	$$\Lambda^*(ds)=d^2s-ds\circ \Gamma =a\otimes ds+ds\otimes b+sc$$
	in $U$, where $\Lambda^*$ is the induced connection on the cotangent bundle $TX^*$ of $X$.
\end{definition}

Now, we are ready to present our uniqueness theorem for holomorphic curves intersecting totally geodesic hypersurfaces.
\begin{theorem}\label{thm:5curve} 
	Let $S_0, S_1, \dots, S_k$ be homogeneous polynomials of degree $d$ in $\mathbb{C}[X_0, \dots, X_k]$ satisfying \begin{equation}\label{eqn:det}
		\det\left(\frac{\partial S_{\mu}}{\partial X_i}\right)_{0\leq \mu, i\leq k}\not\equiv 0.
	\end{equation}
	Assume that $\sigma_j$, $1\leq j\leq q$ are elements of the linear system $Y_{\alpha}=\{\alpha_0S_0+\alpha_1S_1+\dots+\alpha_kS_k=0\}$ such that the hypersurfaces $Y_j=(\sigma_j), 1\leq j\leq q$ are smooth and in a general position. Let $f$ and $g$ be two holomorphic curves from $\C$ into $\mathbb{P}^k$ such that  their images are neither contained in the support of an element of the linear system $Y_{\alpha}$ nor contained in the polar locus of the meromorphic partial projective connection $\Lambda$ induced by (\ref{eqn:2}).
	Suppose that $f(z)=g(z)$ for all $z\in S$, where $$S:=\bigcup_{j=1}^q\left\{f^{-1}(\sigma_j)\cup g^{-1}(\sigma_j)\right\}.$$ Then the following assertions hold:
	\begin{enumerate}[(i)]
		\item\label{thm:5a} if $q> \dfrac{3k+1}{d}+\dfrac{1}{2d}(k-1)k(k+1)(d-1),$ then $f\equiv g$;\\
		\item\label{thm:5b} if $q> 2\left(\dfrac{k+1}{d}+\dfrac{1}{2d}(k-1)k(k+1)(d-1)\right)$, then there exist $k+1$ hypersurfaces $Y_{i_0}, \dots, Y_{i_k}$, $1\leq i_0 <\cdots < i_k\leq q$ such that $$\dfrac{\sigma_{i_0}(f)}{\sigma_{i_0}(g)} =\cdots = \dfrac{\sigma_{i_k}(f)}{\sigma_{i_k}(g)}.$$
	\end{enumerate}
\end{theorem}
\begin{remark}
	The number of hypersurfaces required in Theorem \ref{thm:5curve} becomes much smaller than for the previous results  (cf. \cite{QA17,DR08}).\end{remark}
In the case $d=1$, from Theorem \ref{thm:5curve}(\ref{thm:5b}), we can also recover the uniqueness theorem for linearly non-degenerated holomorphic curves sharing $2k+3$ hyperplanes in $\mathbb{P}^k$ in a general position; this was first obtained by Chen and Yan \cite{CY09}.

Indeed, when $d=1$, let $S_i=X_i$, the meromorphic partial projective connection in $\mathbb{P}^k$ is the flat connection, and it is easy to see that any hyperplane $H$ is an element of the linear system $H_{\alpha}=\{\alpha_0S_0+ \cdots+ \alpha_kS_k=0\}$, and always smooth. We choose $q\geq 2k+3$ hyperplanes $H_i, 1\leq i\leq q$ to be  in a general position. From Theorem \ref{thm:5curve}(\ref{thm:5b}), there exist $k+1$ hyperplanes $H_{i_0}, \dots, H_{i_k}$ such that $$m=\dfrac{H_{i_0}(f)}{H_{i_0}(g)} =\cdots = \dfrac{H_{i_k}(f)}{H_{i_k}(g)}.$$ As $H_1, \dots, H_q$ are in a general position, for any $k+1$ hyperplanes $\{H_{i_0}, \dots, H_{i_k}\}$, the determinant of $(H_{i_0}, \dots, H_{i_k})$ is nonzero, hence  $f_i=mg_i$, and therefore $f\equiv g$.

Recently, Ru and Sibony \cite{RS19} defined a growth index of a holomorphic map $f$ from a disc $\mathbb{D}_R$ centred at zero with radius $R$ to a complex manifold and generalized the classical value distribution theory for holomorphic curves on the whole complex plane. 
\begin{definition}
	Let $M$ be a complex manifold with a positive (1, 1) form $\omega$ of finite volume. Let $0<R\leq \infty$ and let $f:\D_R\rightarrow M$ be a holomorphic map. The growth index of $f$ with respect to $\omega$ is defined as 
	$$c_{f, \omega}:=\inf\left\{c>0\big|\int_0^R\exp(cT_{f, \omega}(r))dr=\infty\right\},$$ where $$T_{f, \omega}(r):=\int_0^r\frac{dt}{t}\int_{|z|<t}f^*\omega, \quad 0<r<R$$ is the characteristic function of $f$ with respect to $\omega$. If $M=\mathbb{P}^k$ and $\omega$ is the Fubini-Study form, then we simply denote $c_{f, \omega}$ by $c_f$.  
\end{definition}
Notice that if $R=\infty$, then $c_{f, \omega}=0$. Therefore, we can also extend Theorem \ref{thm:5curve} to the hyperbolic case.
\begin{theorem}\label{thm:5} 
	Let $S_0, S_1, \dots, S_k$ be homogeneous polynomials of degree $d$ in $\mathbb{C}[X_0, \dots, X_k]$ satisfying \begin{equation}\label{eqn:det}
		\det\left(\frac{\partial S_{\mu}}{\partial X_i}\right)_{0\leq \mu, i\leq k}\not\equiv 0.
	\end{equation}
	Assume that $\sigma_j$, $1\leq j\leq q$ are elements of the linear system $Y_{\alpha}=\{\alpha_0S_0+\alpha_1S_1+\dots+\alpha_kS_k=0\}$ such that the hypersurfaces $Y_j=(\sigma_j), 1\leq j\leq q$ are smooth and in a general position. Let $f$ and $g$ be two holomorphic maps from $\D_R$ into $\mathbb{P}^k$ with $c_f<\infty$ and $c_g<\infty$ such that  their images are neither contained in the support of an element of the linear system $Y_{\alpha}$ nor contained in the polar locus of $\Lambda$.
	Suppose that $f(z)=g(z)$ for all $z\in S$, where $$S:=\bigcup_{j=1}^q\left\{f^{-1}(\sigma_j)\cup g^{-1}(\sigma_j)\right\}.$$ Then the following assertions hold:
	\begin{enumerate}[(i)]
		\item\label{thm:5a} if $q> \dfrac{3k+1}{d}+\dfrac{1}{2d}(k-1)k(k+1)(d-1)+\dfrac{k^2(k+1)^2}{2}\max\{c_f,c_g\},$ then $f\equiv g$;\\
		\item\label{thm:5b} if $q> 2\left(\dfrac{k+1}{d}+\dfrac{1}{2d}(k-1)k(k+1)(d-1)\right)+\dfrac{k^2(k+1)^2}{d}\max\{c_f,c_g\}$, then there exist $k+1$ hypersurfaces $Y_{i_0}, \dots, Y_{i_k}$, $1\leq i_0 <\cdots < i_k\leq q$ such that $$\dfrac{\sigma_{i_0}(f)}{\sigma_{i_0}(g)} =\cdots = \dfrac{\sigma_{i_k}(f)}{\sigma_{i_k}(g)}.$$
	\end{enumerate}
\end{theorem}

To prove the above result, we need to obtain a result analogoues to that of Tiba's SMT for holomorphic mappings on a disk.

\begin{theorem}\label{thm:main}
	Let $X$ be a $k$-dimensional complex projective algebraic manifold with a positive (1, 1) form $\omega$, and let $L_i:=(Q_i)\ (1\leq i\leq q)$ be a smooth holomorphic line bundle on $X$ with the holomorphic section $Q_i$ such that $\sum_{1\leq i\leq q}L_i$ is a simple normal crossing divisor of $X$. Let $\Lambda=\{(\Lambda_j, U_j)\}_{1\leq j\leq N}$ be a meromorphic partial projective connection relative to an affine covering $\{U_j\}$ of $X$, and let $\beta$ be a holomorphic section of the holomorphic line bundle $L$ on $X$ such that $\beta\Lambda_j$ is holomorphic on $U_j$ for all $j$.  Let $f=(f_1, \dots, f_k):\D_R\rightarrow X$ be a holomorphic map with $c_{f, \omega}<\infty$ such that $f(\D_R)$ is not contained in the polar locus $\mathrm{supp}(\Lambda)_{\infty}$ of $\Lambda$ and we have the Wronskian $$W_{\Lambda}(f):=f'\wedge\Lambda_{f'}f'\wedge\dots\wedge\Lambda_{f'}^{(k-1)}f'\not\equiv 0,$$
	where $\Lambda^{(m)}:=\underbrace{\Lambda\circ\dots\circ\Lambda}_{m-\mathrm{times}}$. Suppose that $\mathrm{supp}(Q_i)$ is not contained in $\mathrm{supp}(\Lambda)_{\infty}$ and $L_i$ is totally geodesic with respect to $\Lambda$ for $1\leq i\leq q$. Then 
	$$\sum_{1\leq i\leq q}T_f(r, L_i)+T_f(r, K_X)-\frac{1}{2}k(k-1)T_f(r, L)\leq \sum_{1\leq i\leq q}N_k(r, f^*(Q_i))+S_f(r),$$ where, for any $\epsilon>0$, $$S_f(r)=O(\log^+r+\log^+T_f(r)+\sum_{1\leq m\leq k}\log^+T_{f_m}(r))+\frac{Nk^2(k+1)}{2}(1+\epsilon)(c_{f, \omega}+\epsilon)T_f(r)\ \|.$$ Here $\|$ means that the inequality holds outside a set $E\subset(0, R)$ with $\displaystyle\int_E\exp((c_{f, \omega}+\epsilon)T_f(r))dr<\infty$. 
\end{theorem}  

If $X=\mathbb{P}^k$, then we have the following result: 
\begin{theorem}\label{thm:2} 
	For each $j=1, \dots, q$, let $\sigma_j$ be an element of the linear system $|\{S_0, \dots, S_k\}|$ such that the divisor $(\sigma_j)$ is smooth and $\sum_{1\leq j\leq q}(\sigma_j)$ is a simple normal crossing divisor.
	
	Let $f=[f_0: f_1:\cdots:f_k]:\D_R\rightarrow \mathbb{P}^k$ be a non-constant holomorphic map with $c_f<\infty$ such that its image is neither contained in the support of an element of linear system $|\{S_0, \dots, S_n\}|$ nor contained in the polar locus of $\Lambda$. Then we have 
	$$\left(q-\frac{k+1}{d}-\frac{1}{2d}(k-1)k(k+1)(d-1)\right)T_f(r, dH) \leq \sum_{1\leq j\leq q}N_k(r, f^*(\sigma_j))+S_f(r),$$ where $H$ is a hyperplane bundle on $\mathbb{P}^k$, and, for any $\epsilon>0$, $$S_f(r)=O(\log^+T_f(r) +\log^+r)+\frac{k^2(k+1)^2}{2}(1+\epsilon)(c_f+\epsilon)T_f(r)\ \|.$$ 
\end{theorem}
\begin{remark} When the degree of $S_i$ is equal to 1, one can choose $(S_i)$ to be the coordinate hyperplane, and then the meromorphic connection constructed becomes the flat connection. Thus any hyperplane is a totally geodesic hypersurface with respect to the flat connection on $\mathbb{P}^k$ and the holomorphic curve $f$ is linearly non-degenerated if and only if $W_{\nabla}(f)\not\equiv 0$. Here $W_{\nabla}(f)$ is actually equivalent to the classical Wronskian of $f$. Hence Theorem \ref{thm:2} can be regarded as a generalization of Cartan's Second Main Theorem.
\end{remark}
\section{Notations and Prelimiraries}\label{sec:2}
In this section, we introduce basic definitions and results for the Nevanlinna theory.

Let $f:\D_R\rightarrow \mathbb{P}^k$ be a nonconstant holomorphic map. Let $E=\sum m_jP_j$ be an effective divisor on $\mathbb{D}_R$, where $P_j$ is a set of discrete points in $\D_R$ and $m_j$ are positive integers. For $0<r<R$, put $n_k(r, E)=\sum_{|P_j|<r}\min\{k, m_j\}$. We define the counting function of $E$ by $$N_k(r, E)=\int_0^r \frac{n_k(t, E)}{t}dt.$$ 
Let $X$ be a complex projective algebraic manifold with dimension $k$ (e.g. $X=\mathbb{P}^k$), and let $D$ be an effective divisor on $X$. Put $L=\mathcal{O}(D)$, where $\mathcal{O}(D)$ denotes the line bundle associated with $D$. Let $\sigma$ be a holomorphic section of $L$ such that $D=\{\sigma=0\}$. We define the proximity function of $f$ with respect to $D$ under the assumption that $f(\D_R)\not\subset D$, by $$m_f(r, D)=\int_0^{2\pi}\log\frac{1}{\|\sigma(f(re^{i\theta}))\|_L}\frac{d\theta}{2\pi},$$ where $\|\cdot\|_L$ is a hermitian metric in $L$. In particular, if $\sigma$ is a homogeneous polynomial of degree $d$, then the proximity function could be defined by $$m_f(r, D)=\int_0^{2\pi}\log\frac{\|f(re^{i\theta})\|^d}{|\sigma(f(re^{i\theta}))|}\frac{d\theta}{2\pi}, $$ where $\|f(re^{i\theta)}\|=\max \{ |f_0(re^{i\theta})|, \dots , |f_k(re^{i\theta})|\}$, and $f=(f_0, \dots, f_k)$ a reduced representation.
The integrated counting function and truncated counting functions of $f$ with respect to $D$ are defined, under the assumption that $f(\D_R)\not\subset D$, by
\begin{eqnarray*}
	N(r, f^*D)=\int_0^r\frac{n(t,  f^*D)-n(0,  f^*D)}{t}dt+n(0,  f^*D)\log r,\\[0.1cm]
	N_k(r, f^*D)=\int_0^r\frac{n_k(t,  f^*D)-n_k(0,  f^*D)}{t}dt+n_k(0,  f^*D)\log r,
\end{eqnarray*}
where $n(t, f^*D)$ equals the number of points of $f^{-1}(D)$ in the disc $|z|<t$, counting multiplicity, $n(0,f^*D)=\lim_{t\rightarrow 0}n(t, f^*D)$, and $n_k(t, f^*D)=\min\{k, n(t, f^*D)\}$. 
Similarly, for a set $S\subset\mathbb{D}_R$ with $ \#(S\cap\{ |z|<r\})<\infty$ for every $0<r<R$, we denote by $n(r, S)$  the number of elements of $S$ in the disk with a center at zero and radius $r$, and $$N(r, S)=\int_0^r\frac{n(t,  S)-n(0, S)}{t}dt+n(0, S)\log r.$$ For a line bundle associated with $D$, we define the characteristic function of $T_f(r, \mathcal{O}(D))$ by $$T_f(r, \mathcal{O}(D))=\int_0^r\frac{dt}{t}\int_{|z|\leq t} f^*c_1(\mathcal{O}(D)), $$ where $c_1(\mathcal{O}(D))$ is the Chern form of $\mathcal{O}(D)$. If the line bundle is a hyperlane line bundle $H$, we set $$T_f(r)=T_f(r, H)=\int_0^r\frac{dt}{t}\int_{|z|\leq t} f^*\omega_{FS}=\frac{1}{2\pi}\int_0^{2\pi}\log \|f(re^{i\theta})\|d\theta+O(1),$$
where $\omega_{FS}$ is the Fubini-Study form on $\mathbb{P}^k$.
\section{Proof of Theorem \ref{thm:main} and \ref{thm:2}}
To prove our results, we need the following lemmata provided by Ru and Sibony \cite{RS19}:

\begin{lemma}[Calculus Lemma \cite{RS19}]\label{lem:Cl}
	Let $0<R\leq\infty$ and let $\gamma(r)$ be a nonnegative function defined on $(0, R)$ with $\int_0^R\gamma(r)dr=\infty$. Let $h$ be a nondecreasing function of class $C^1$ defined on $(0, R)$. Assume that $\lim_{r\rightarrow R}h(r)=\infty$ and $h(r_0)\geq c>0$. Then, for every $0<\delta<1$, the inequality $$h'(r)\leq h^{1+\delta}(r)\gamma(r)$$ holds for all $r\in(0, R)$ outside a set $E$ with $\int_E\gamma(r)dr<\infty$.
\end{lemma}

\begin{lemma}[Logarithmic Derivative Lemma \cite{RS19}]\label{lem:LDL}
	Let $R$ and $\gamma(r)$ be as in Lemma \ref{lem:Cl}. Let $f$ be a meromorphic function on $\D_R$. Then, for $l\geq 1$ and $\delta>0$, the inequality 
	\begin{align*}
		\int_0^{2\pi}&\log^+\left|\frac{f^{(l)}}{f}(re^{i\theta})\right|\frac{d\theta}{2\pi}\leq (1+\delta)l\log\gamma(r)+\delta l\log r\\
		&+O(\log T_f(r)+\log\log\gamma(r)+\log\log r)
	\end{align*} holds outside a set $E\subset (0, R)$ with $\int_E\gamma(r)dr<\infty$.
\end{lemma}

The following lemma was obtained by Demaily \cite{Dem97}:
\begin{lemma}[\cite{Dem97}, \cite{Tib12}]\label{lem:1}
	Assume that $D=(s)$ is totally geodesic with respect to $\Lambda$ on $U$. Let $\beta$ be a holomorphic function on $U$ such that $\beta a, \beta b, \beta c$ are holomorphic forms, where $a, b, c$ are defined as in Definition \ref{def:tgh}. Let $V$ be a domain in $\C$. Let $f:V\rightarrow U$ be a holomorphic map. Then, for $m\in\N$, we have that $$(s\circ f)^{(m)}=\gamma_ms\circ f+\sum_{0\leq l\leq m-2}\gamma_{l,m}ds\cdot\Lambda_{f'}^{(l)}f'+ds\cdot\Lambda_{f'}^{(m-1)}f'$$ in $V$. Here $\gamma_m$ and $\gamma_{l, m}$ are meromorphic on $V$ such that $\beta^{m-1}(f)\gamma_m$ and $\beta^{m-l-1}(f)\gamma_{l,m}$ are holomorphic on $V$.
\end{lemma}
\subsection{Proof of Theorem \ref{thm:main}} We follow the argument of \cite{Tib12}. Let $\{V_j\}_{1\leq j\leq N}$ be an open covering of $X$ such that the topological closure $\bar{V}_j$ is contained in $U_j$ and $\bar{V}_j$ is compact. Then, we have a partition of unity $\{\phi_j\}_{1\leq j\leq N}$ subordinate to the covering $\{V_j\}$. Take holomorphic function $z_1, \dots, z_k$ on $U_j$ such that $dz_1, \dots, dz_k$ are linearly independent and 
$$U_j\cap\bigcup_{i=1}^q\mathrm{supp}(Q_i)=\{w\in U_j|z_1(w)\cdots z_p(w)=0\}$$ for some $p,\ 0\leq p\leq k$. Set $$f_l=z_l\circ f, \quad (\Lambda_{f'}^{(m)}f')_l=dz_l\cdot\Lambda_{f'}^{(m)}f'.$$ Then we have 
\begin{align*}
	\phi_j(f)&\log^+\frac{\|W_{\Lambda}(f)\|_{\wedge^nTX}\|\beta(f)\|^{k(k-1)/2}_L}{\prod_{i=1}^q\|Q_i(f)\|_{L_j}}\\
	&=\phi_j(f)\log^+\big(\Phi_j(f)\|\beta(f)\|_L^{k(k-1)/2}\\
	&\times\begin{vmatrix}
		\frac{f_1'}{f_1} & \cdots & \frac{f_p'}{f_p} & f_{p+1} &\cdots &f_k' \\ 
		\frac{((\Lambda_j)_{f'}f')_1}{f_1} & \cdots & \frac{((\Lambda_j)_{f'}f')_p}{f_p} & ((\Lambda_j)_{f'}f')_{p+1} &\cdots &((\Lambda_j)_{f'}f')_k \\ 
		\cdot & \cdot & \cdot & \cdot &\cdot &\cdot\\
		\cdot & \cdot & \cdot & \cdot &\cdot &\cdot \\ 
		\cdot & \cdot & \cdot & \cdot &\cdot &\cdot\\
		\frac{((\Lambda_j)^{(k-1)}_{f'}f')_1}{f_1} & \cdots & \frac{((\Lambda_j)^{(k-1)}_{f'}f')_p}{f_p} & ((\Lambda_j)^{(k-1)}_{f'}f')_{p+1} &\cdots &((\Lambda_j)^{(k-1)}_{f'}f')_k
	\end{vmatrix}\big)
\end{align*}
on $f^{-1}(U_j)$, where $\Phi_j$ is a $C^{\infty}$-function on $U_j$. By Lemma \ref{lem:1},
$$((\Lambda_j)^{(m)}_{f'}f')_i=\sum_{0\leq l\leq m+1}a_{i, m, l}(z)\frac{d^lf_i}{dz^l}(z)$$ for $1\leq i\leq p$, where $a_{i,m,l}$ are meromorphic functions on $f^{-1}(U_j)$ such that $a_{i, m,l}(\beta\circ f(z))^m$ is a holomorphic function. Hence it follows that, for $0<r<R$,

\begin{align*}
	K_j:&=\int_0^{2\pi}\phi_j(f)\log^+\frac{\|W_{\Lambda}(f)\|_{\wedge^nTX}\|\beta(f)\|^{k(k-1)/2}_L}{\prod_{i=1}^q\|Q_i(f)\|_{L_j}}(re^{i\theta})\frac{d\theta}{2\pi}\\
	&\leq \int_0^{2\pi}\Phi(f)(re^{i\theta})\frac{d\theta}{2\pi}+\sum_{1\leq m\leq p}\sum_{1\leq l\leq k}\int_0^{2\pi}\log^+\left|\frac{f^{(l)}_m}{f_m}(re^{i\theta})\right|\frac{d\theta}{2\pi}\\
	&+\sum_{p+1\leq m\leq k}\sum_{1\leq l\leq k}\int_0^{2\pi}\Psi(f(re^{i\theta}))\log^+|f^{(l)}_m(re^{i\theta})|\frac{d\theta}{2\pi},
\end{align*}
where $\Phi$ and $\Psi$ are bounded $C^{\infty}$-functions on $X$. By using the Logarithmic Derivative Lemma (Lemma \ref{lem:LDL}) with $\gamma(r):=\exp((c_{f, \omega}+\epsilon)T_f(r))$, it follows that 
\begin{align*}
	&\int_0^{2\pi}\log^+\left|\frac{f^{(l)}_m}{f_m}(re^{i\theta})\right|\frac{d\theta}{2\pi}\leq S_f(r, l)\ \|,\\
	&\int_0^{2\pi}\Psi(f(re^{i\theta}))\log^+|f^{(l)}_m(re^{i\theta})|\frac{d\theta}{2\pi}\\
	&\leq \int_0^{2\pi}\Psi(f(re^{i\theta}))\log^+|f_m'(re^{i\theta})|\frac{d\theta}{2\pi}+\int_0^{2\pi}\log^+\left|\frac{f^{(l)}_m}{f'_m}(re^{i\theta})\right|\frac{d\theta}{2\pi}\\
	&\leq \int_0^{2\pi}\Psi(f(re^{i\theta}))\log^+|f_m'(re^{i\theta})|\frac{d\theta}{2\pi}+S_f(r, l-1),
\end{align*}
where $$S_f(r, l)=(1+\epsilon)l(c_{f, \omega}+\epsilon)T_f(r)+\epsilon l\log r+O(\log T_{f_m}(r))\ \|.$$
Also, 
\begin{align*}
	\int_0^{2\pi}\Psi(f(re^{i\theta}))\log^+|f_m'(re^{i\theta})|\frac{d\theta}{2\pi}&=\frac{1}{2} \int_0^{2\pi}\Psi(f(re^{i\theta}))\log^+|f_m'(re^{i\theta})|^2\frac{d\theta}{2\pi}\\
	&\leq  \frac{1}{2}\int_0^{2\pi}\log^+\|f'(re^{i\theta})\|_{TX}^2\frac{d\theta}{2\pi}+O(1),
\end{align*}
where $\|\cdot\|_{TX}$ is a hermitian metric of $TX$.

By the Calculus Lemma  with $\gamma(r):=\exp((c_{f, \omega}+\epsilon)T_f(r))$ and the concavity of $\log$, we have that  
\begin{align*}
	\frac{1}{2}&\int_0^{2\pi}\log^+\|f'(re^{i\theta})\|_{TX}^2\frac{d\theta}{2\pi}\\
	&\leq \frac{1}{2}\int_0^{2\pi}\log(\|f'(re^{i\theta})\|_{TX}^2+1)\frac{d\theta}{2\pi}\\
	&\leq \frac{1}{2}\log(1+\int_0^{2\pi}\|f'(re^{i\theta})\|_{TX}^2\frac{d\theta}{2\pi})+O(1)\\
	&\leq \frac{1}{2}\log(1+\frac{1}{2\pi r}\frac{d}{dr}\int_{\D_r}\|f'(z)\|_{TX}^2\frac{\sqrt{-1}}{2}dz\wedge d\bar{z})+O(1)\\
	&\leq \frac{1}{2}\log(1+\frac{1}{2\pi r}\left(\int_{\D_r}\|f'(z)\|_{TX}^2\frac{\sqrt{-1}}{2}dz\wedge d\bar{z}\right)^{1+\epsilon}\gamma(r))+O(1)\ \|\quad (\mathrm{Lemma}\ \ref{lem:Cl})\\
	&=\frac{1}{2}\log(1+\frac{r^{\epsilon}}{2\pi}\left(\frac{d}{dr}\int_1^r\frac{dt}{t}\int_{\D_t}\|f'(z)\|_{TX}^2\frac{\sqrt{-1}}{2}dz\wedge d\bar{z}\right)^{1+\epsilon}\gamma(r))+O(1)\ \|\\
	&\leq \frac{1}{2}\log(1+\frac{r^{\epsilon}}{2\pi}T_f(r)^{(1+\epsilon)^2}\gamma(r)^{2+\epsilon})+O(1)\ \|\quad\ (\mathrm{Lemma}\ \ref{lem:Cl})\\
	&\leq \epsilon\log^+ r+(1+\epsilon)^2\log^+T_f(r)+(1+\epsilon)(c_{f, \omega}+\epsilon)T_f(r)+O(1)\ \|.
\end{align*}
Therefore, 
\begin{align*}
	K_j\leq & O(1)+\sum_{1\leq m\leq p}\sum_{1\leq l\leq k}S_f(r, l)+\sum_{p+1\leq m\leq k}\sum_{1\leq l\leq k}S_f(r, l-1)\\
	&+\sum_{p+1\leq m\leq k}\sum_{1\leq l\leq k}(\epsilon\log^+ r+(1+\epsilon)^2\log^+T_f(r)+(1+\epsilon)(c_{f, \omega}+\epsilon)T_f(r))\\
	&\leq O(\log^+r+\log^+T_f(r)+\sum_{1\leq m\leq k}\log^+T_{f_m}(r))+\frac{k^2(k+1)}{2}(1+\epsilon)(c_{f, \omega}+\epsilon)T_f(r)\ \|.
\end{align*}
Then we have that
\begin{align*}
	\int_0^{2\pi}&\log^+\frac{\|W_{\Lambda}(f)\|_{\wedge^nTX}\|\beta(f)\|^{k(k-1)/2}_L}{\prod_{i=1}^q\|Q_i(f)\|_{L_j}}(re^{i\theta})\frac{d\theta}{2\pi}=\sum_{j=1}^NK_j\\
	&\leq O(\log^+r+\log^+T_f(r)+\sum_{1\leq m\leq k}\log^+T_{f_m}(r))+\frac{Nk^2(k+1)}{2}(1+\epsilon)(c_{f, \omega}+\epsilon)T_f(r)\ \|.
\end{align*}

We denote by $\mathrm{ord}_z(Q_j\circ f)$ the order of zeros of $Q_j\circ f$ at the point of $z\in\D_R$, and by $\mathrm{ord}_z\beta(f)^{k(k-1)/2}W_{\Lambda}(f)$ the order of zero of $\beta(f)^{k(k-1)/2}W_{\Lambda}(f)$ at the point $z\in\D_R$. If $\mathrm{ord}_z(Q_j\circ f)\geq k+1$ for $z\in\D_R$, then 
$$\mathrm{ord}_z\beta(f)^{k(k-1)/2}W_{\Lambda}(f)\geq \mathrm{ord}_z(Q_j\circ f)-k.$$
By Lemma \ref{lem:1} and the First Main Theorem, we have that
\begin{align*}\sum_{1\leq i\leq q}&T(r, L_i)-\sum_iN_k(r, f^*(Q_i))-T_f(r, \wedge^nTX)-\frac{k(k-1)}{2}T_f(r, L)\\
	&\leq \int_0^{2\pi}\log^+\frac{\|W_{\Lambda}(f)\|_{\wedge^nTX}\|\beta(f)\|^{k(k-1)/2}_L}{\prod_{i=1}^q\|Q_i(f)\|_{L_j}}(re^{i\theta})\frac{d\theta}{2\pi}\leq S_f(r).
\end{align*}
This completes the proof of Theorem \ref{thm:main}.

\subsection{Proof of Theorem \ref{thm:2}} Let $[X_0: \cdots:X_k]$ be a homogeneous coordinate system of $\mathbb{P}^k$. Then by the same method used in Section 3 of Tiba \cite{Tib12}, one can construct the meromorphic partial projective connection $\Lambda=\{(\Lambda_j, U_j)\}_{0\leq j\leq k}$ on $\mathbb{P}^k$, where $U_j=\{[X_0:\cdots:X_k]\in\mathbb{P}^k|X_j\neq 0\}$. 
By Crammer's rule, the solutions are of the form $$\widetilde{\Gamma}_{i,j}^{\lambda}=\dfrac{\delta_{i,j}^{\lambda}}{\delta},$$ where $\delta=\det(\partial S_{\mu}/\partial X_{\lambda})_{0\leq\mu, \lambda\leq k}$, and $\delta_{i,j}^{\lambda}$ is the determinant replacing the column of index $\lambda$ in $\delta$ by the column $\dfrac{\partial^2 S_{\mu}}{\partial X_i\partial X_j}$ for $0\leq\mu\leq k$. Since $\partial S_{\mu}/\partial X_{\lambda}$ is a homogeneous polynomial of degree $d-1$, $\delta$ is a homogeneous polynomial of degree $(k+1)(d-1)$. This implies that the degree of the polar divisor of each $\Lambda_j$ is less than or equal to $(k+1)(d-1)$, hence $T(r, L)\leq(k+1)(d-1)T_f(r)$. In this case, $N=k+1$ and $T_{f_i/f_0}(r)\leq T_f(r)+O(1)$. Thus, Theorem \ref{thm:2} follows from Theorem \ref{thm:main}, the First Main Theorem and $K_{\mathbb{P}^k}=-(k+1)H$.

\section{Proof of Theorem \ref{thm:5}} 

To prove Theorem \ref{thm:5}, we need the following important proposition from \cite{DR08}: 

\begin{proposition}[\cite{DR08}]\label{pro:1} Let $\sigma_j, 1\leq j\leq q$ be the smooth hypersurfaces defined in Theorem \ref{thm:5}.
	Let $f$ and $g$ be two holomorphic maps from $\mathbb{D}_R$ into $\mathbb{P}^k$ with $c_f<\infty$ and $c_g<\infty$. Suppose that $f(z)=g(z)$ for all $z\in S$, where $$S:=\bigcup_{j=1}^q\left\{f^{-1}(\sigma_j)\cup g^{-1}(\sigma_j)\right\}.$$ 
	If $f\not\equiv g$, we have that
	$$N(r, S)\leq T_f(r)+T_g(r)+O(1).$$ 
\end{proposition}
\begin{proof} Let $\pi_1, \pi_2$ be the projective maps from $\mathbb{P}^k\times\mathbb{P}^k$ into the first $\mathbb{P}^k$ and the second one, respectively. If $f\not\equiv g$, then from the proof of Lemma \ref{lem:DMR} it is not hard to see that the diagonal $(f\times g)(S)$ is in $\mathrm{Supp}(s)$, but the image $(f\times g)(\mathbb{D}_R)$ is not contained in $\mathrm{Supp}(s)$, where $s$ is a section of $H':=\pi_1H\otimes\pi_2 H$ determined by a non-constant polynomial with complex coefficients $$P([z_0, \dots, z_k], [w_0, \dots,w_k])=\sum_{0\leq m<l\leq k}a_{ml}(z_mw_l-z_lw_m).$$ 
	Thus, we have that $$N(r, S)\leq N(r, (f\times g)^*(s)).$$ By the First Main Theorem for Line Bundles, $$N(r, (f\times g)^*(s))\leq T_{f\times g}(r, H')+O(1)\leq T_f(r)+T_g(r)+O(1).$$ Hence, we obtain the result.
\end{proof}

\subsection{Proof of Theorem \ref{thm:5}(\ref{thm:5a})}
The proof is by contradiction. Suppose that $f\not\equiv g$. Theorem \ref{thm:2} gives that 
\begin{eqnarray*}
	\left(q-\frac{k+1}{d}-\frac{1}{2d}(k-1)k(k+1)(d-1)\right)T_f(r) &\leq& d^{-1}\sum_{1\leq i\leq q}N_k(r, f^*(\sigma_i))+S_f(r), \\[0.6cm]
	\left(q-\frac{k+1}{d}-\frac{1}{2d}(k-1)k(k+1)(d-1)\right)T_g(r) & \leq& d^{-1}\sum_{1\leq i\leq q}N_k(r, g^*(\sigma_i))+S_g(r).
\end{eqnarray*}
Thus, 
\begin{eqnarray*}
	\left(q-\frac{k+1}{d}-\frac{1}{2d}(k-1)k(k+1)(d-1)\right)(T_f(r)+T_g(r)) \\[0.5cm]
	\leq d^{-1}\sum_{1\leq i\leq q}\left(N_k(r, f^*(\sigma_i))+N_k(r, g^*(\sigma_i))\right)+(S_f(r)+S_g(r)).
\end{eqnarray*}
From the condition $f(z)=g(z)$ on $S$, we have that
\begin{eqnarray*}
	\sum_{1\leq i\leq q}(n_k(r, f^*(\sigma_i))+n_k(r, g^*(\sigma_i))) &\leq & k\sum_{1\leq i\leq q}(n_1(r, f^*(\sigma_i))+n_1(r, g^*(\sigma_i)))\\
	&\leq & 2k n(r, S).
\end{eqnarray*}
Hence, 
\begin{eqnarray*}
	\left(q-\frac{k+1}{d}-\frac{1}{2d}(k-1)k(k+1)(d-1)\right)(T_f(r)+T_g(r)) \\[0.5cm]
	\leq \frac{2k}{d}N(r, S)+(S_f(r)+S_g(r)).
\end{eqnarray*}
Let $c_{f, g}=\max\{c_f, c_g\}$. Applying Proposition \ref{pro:1} gives 
\begin{align*}
	&\left(q-\frac{k+1}{d}-\frac{1}{2d}(k-1)k(k+1)(d-1)\right)(T_f(r)+T_g(r))\\
	&\leq \frac{2k}{d}(T_f(r)+T_g(r))+(S_f(r)+S_g(r))\\
	&\leq \frac{2k}{d}(T_f(r)+T_g(r))+\frac{k^2(k+1)^2}{2}(1+\epsilon)(c_{f, g}+\epsilon)(T_f(r)+T_g(r))\\
	&\quad \quad +O(\log^+r+\log^+T_f(r)+\log^+T_g(r))\quad \|.
\end{align*}
Thus $$q\leq \dfrac{3k+1}{d}+\dfrac{1}{2d}(k-1)k(k+1)(d-1)+\frac{k^2(k+1)^2}{2}c_{f,g},$$
which is in contradiction to  the assumption that $$q> \dfrac{3k+1}{d}+\dfrac{1}{2d}(k-1)k(k+1)(d-1)+\dfrac{k^2(k+1)^2}{2}\max\{c_f, c_g\}.$$ Hence $f\equiv g$. 
\subsection{Proof of Theorem \ref{thm:5}(\ref{thm:5b})}
We follow the method of Chen and Yan \cite{CY09}.  Suppose that the assertion does not hold. By changing indices if necessary, we may assume that 
\begin{eqnarray*} 
	\underbrace{\frac{\sigma_{1}(f)}{\sigma_{1}(g)}\equiv \cdots \equiv \frac{\sigma_{v_1}(f)}{\sigma_{v_1}(g)}}_{\mbox{Group 1}} \not\equiv\underbrace{\frac{\sigma_{v_1+1}(f)}{\sigma_{v_1+1}(g)} \equiv \cdots \equiv \frac{\sigma_{v_2}(f)}{\sigma_{v_2}(g)}}_{\mbox{Group 2}}
	\not\equiv \cdots \not\equiv \underbrace{\frac{\sigma_{v_{s-1}+1}(f)}{\sigma_{v_{s-1}+1}(g)} \equiv\cdots \equiv \frac{\sigma_{v_s}(f)}{\sigma_{v_s}(g)}}_{\mbox{Group s}},
\end{eqnarray*}
where $v_s=q$. 

Since the assertion (\ref{thm:5b})  does not hold, the number of elements of each group is at most $k$. For each $1\leq i\leq q$, we set 
$$ p(i)=\left\{
\begin{array}{lcl}
	i+k       &      & {i+k\leq q}\\
	i+k-q     &      & {i+k>q}
\end{array} \right. $$
and  $P_i=\sigma_i(f)\sigma_{p(i)}(g)-\sigma_i(g)\sigma_{p(i)}(f).$ Then $\dfrac{\sigma_i(f)}{\sigma_i(g)}$ and $\dfrac{\sigma_{p(i)}(f)}{\sigma_{p(i)}(g)}$ are from two distinct groups, and hence $P_i\not\equiv 0$ for every $1\leq i\leq q$. Since $f(z)=g(z)$ for all $z\in S$, note that $ n_f(r, \sigma_j)=n(r, 0, \sigma_j(f))$ for all $j=1,\dots, q$. For convenience, we denote by $n_E(r, 0, f)$ the number of zeros of $f$ in $\mathbb{D}_r\cap E$. Recall that $$S=\bigcup_{j=1}^q\left\{f^{-1}(\sigma_j)\cup g^{-1}(\sigma_j)\right\}.$$  Let $S_1=S\setminus\{\sigma_i(f)\sigma_{p(i)}(f)=0\}$, 
we consider the following cases:

\noindent\textbf{Case 1:} If $z\in S_1$, then $\sigma_j(f)(z)\neq0$ for $j=i, p(i)$, and hence
$$n_{S_1}(r, 0, P_i)\geq \sum_{j=1\ j\neq i, p(i)}^q \min\{1, n_{S_1}(r, 0, \sigma_j(f))\}.$$ 

\noindent\textbf{Case 2:} If $z\not\in S_1$, then $\sigma_j(f)(z)\neq 0$ for $j\neq i, p(i)$ and $\sigma_i(f)(z)\sigma_{p(i)}(f)(z)=0$, hence the form of $P_i$ gives that
\begin{eqnarray*}n_{S_1^c}(r, 0, P_i)
	&\geq &\min\{n_{S_1^c}(r, 0, \sigma_i(f)), n_{S_1^c}(r, 0,  \sigma_i(g))\}\\
	&\ \ &+\min\{n_{S_1^c}(r, 0, \sigma_{p(i)}(f)), n_{S_1^c}(r, 0,  \sigma_{p(i)}(g))\},
\end{eqnarray*}
where $S_1^c$ means the complement of $S_1$ in $\mathbb{D}_r$.

Combining Case 1, Case 2 and the fact that $n_f(r, \sigma_i)=n_{S_1}(r, 0, \sigma_i(f))+n_{S_1^c}(r, 0, \sigma_i(f))$, as well as $\min\{a, b\}\geq\min\{a, k\}+\min\{b, k\}-k$, we deduce that 
\begin{eqnarray*}n_{P_i}(r, 0)
	&\geq &\min\{n_f(r, \sigma_i), n_g(r, \sigma_i)\}+\min\{n_f(r, \sigma_{p(i)}), n_g(r, \sigma_{p(i)})\}\\
	&\ \ &+\sum_{j=1\ j\neq i, p(i)}^q \min\{1, n_f(r, \sigma_j)\}\\
	&\geq &\sum_{j=i, p(i)}\bigg(\min\{n_f(r, \sigma_j), k\}+\min\{n_g(r, \sigma_j), k\} 
	-k\min\{n_f(r, \sigma_j), 1\}\bigg)\\ 
	&\ \ &+\sum_{j=1\ j\neq i, p(i)}^q \min\{1, n_f(r, \sigma_j)\}.
\end{eqnarray*}
Integrating both sides of this inequality, we obtain that
\begin{eqnarray*}N_{P_i}(r, 0)
	&\geq &\sum_{j=i, p(i)}\bigg(N_k(r, f^*(\sigma_j))+N_k(r, g^*(\sigma_j)) -kN_1(r, f^*(\sigma_j))\bigg)\\
	&\ \ & +\sum_{j=1\ j\neq i, p(i)}^q N_1(r, f^*(\sigma_j)).
\end{eqnarray*}
Notice that $$\frac{P_i}{\sigma_i(g)\sigma_i(f)}=\frac{\sigma_{p(i)}(g)}{\sigma_i(g)}-\frac{\sigma_{p(i)}(f)}{\sigma_i(f)}$$ and that $\sigma_i(g),\ \sigma_i(f)$ are entire functions. Then, for $0<r<R$, $$N_{P_i}(r, 0)\leq T\left(r, \frac{\sigma_{p(i)}(g)}{\sigma_i(g)}-\frac{\sigma_{p(i)}(f)}{\sigma_i(f)}\right)\leq d(T_f(r)+T_g(r))+O(1).$$
Thus, we have that
\begin{eqnarray*}\sum_{i=1}^qd(T_f(r) &+&T_g(r))+O(1)
	\geq \sum_{i=1}^q\sum_{j=i, p(i)}\bigg(N_k(r, f^*(\sigma_j))+N_k(r, g^*(\sigma_j))\\ 
	&\ \ &-kN_1(r, f^*(\sigma_j))\bigg)+\sum_{i=1}^q\sum_{j=1\ j\neq i, p(i)}^q N_1(r, f^*(\sigma_j))\\
	&\geq& 2\sum_{j=1}^q\bigg(N_k(r, f^*(\sigma_j))+N_k(r, g^*(\sigma_j))\bigg)
	+(q-2(k+1))\sum_{j=1}^qN_1(r, f^*(\sigma_j)).
\end{eqnarray*}
By Theorem \ref{thm:2} and $q\geq 2(k+1)$, it follows that 
\begin{align*}
	&qd(T_f(r) +T_g(r))+2(S_f(r)+S_g(r))+O(1)\\
	&\geq 2d\left(q-\frac{k+1}{d}-\frac{1}{2d}(k-1)k(k+1)(d-1)\right)\left(T_f(r)+T_g(r)\right).
\end{align*}
Since $S_f(r)+S_g(r)\leq O(\log^+r+\log^+T_f(r)+\log^+T_g(r))+\dfrac{k^2(k+1)^2}{2}(1+\epsilon)(c_{f, g}+\epsilon)(T_f(r)+T_g(r))$, it follows that 
$$q\leq2\left(\frac{k+1}{d}+\frac{1}{2d}(k-1)k(k+1)(d-1)\right)+\frac{k^2(k+1)^2}{d}\max(c_{f},c_{g}),$$ which is in contradiction to the assumption in Theorem \ref{thm:5}(\ref{thm:5b}).
This completes the proof.

%\acknowledgements{\rm The author of this paper  would like to thank both Dr. Sun Yunzhi and Dr.Long Jing from Acta Mathematica Sinica and Mrs. Liu Feifei for their help during writting  actams.cls.}


\begin{thebibliography}{99}
	\bibitem{QA17} An D P, Quang S D. Second main theorem and unicity of meromorphic mappings for hypersurfaces
	in projective varieties. Acta Math Vietnam, 2017, \textbf{42}(3):455–470
	\bibitem{AP09} An T T H, Phuong H T. An explicit estimate on multiplicity truncation in the second main theorem
	for holomorphic curves encountering hypersurfaces in general position in projective space. Houston J
	Math, 2009, \textbf{35}(3):775–786
	\bibitem{CY09} Chen Z H, Yan Q M. Uniqueness theorem of meromorphic mappings into $\mathbb{P}^n(\mathbb{C})$ sharing $2n + 3$
	hyperplanes regardless of multiplicities. Int J Math, 2009, \textbf{20}:717–726
	\bibitem{Dem97} Demailly J P. Algebraic criteria for Kobayashi hyperbolic projective varieties and jet differentials. Vol 62, In Proc
	Sympos Pure Math, Amer Math Soc, Providence, 1997:285–360
	\bibitem{DQT12} Dethloff G, Quang S D, Tan T V. A uniqueness theorem for meromorphic mappings with two
	families of hyperplanes. Proc Amer Math Soc, 2012, \textbf{140}:189–197
	\bibitem{DR08} Dulock M, Ru M. A uniqueness theorem for holomorphic curves sharing hypersurfaces. Complex Var
	Elliptic Equ, 2008, \textbf{53}(8):797–802
	\bibitem{DR10} Dulock M, Ru M. Uniqueness of holomorphic curves into Abelian varieties. Trans Amer Math Soc,
	2010, \textbf{363}(1):131–142
	\bibitem{Fuj75} Fujimoto H. The uniqueness problem of meromorphic maps into the complex projective space. Nagoya
	Math J, 1975, \textbf{58}:1–23
	\bibitem{Ru09} Ru M. Holomorphic curves into algebraic varieties. Ann of Math, 2009, \textbf{169}:255–267
	\bibitem{RS19} Ru M, Sibony N. The Second Main Theorem in the hyperbolic case. Math Ann, 2020, \textbf{377}:759–795
	\bibitem{Siu87} Siu Y T. Defect relations for holomorphic maps between spaces of different dimensions. Duke Math J,
	1987, \textbf{55}(1):213–251
	\bibitem{Smil83} Smiley L. Geometric conditions for unicity of holomorphic curves. Contemp Math, 1983, \textbf{25}:149–154
	\bibitem{Tib12} Tiba Y. The second main theorem of hypersurfaces in the projective space. Math Z, 2012, \textbf{272}:1165–1186
\end{thebibliography}
\end{document}